\documentclass[12pt,english]{amsart}
\usepackage{amsmath, amsthm, latexsym, amssymb}
\usepackage{amsfonts}
\usepackage{xypic}
\usepackage{epsfig}
\input xy
\xyoption{all}

\newcommand{\shrinkmargins}[1]{
  \addtolength{\textheight}{#1\topmargin}
  \addtolength{\textheight}{#1\topmargin}
  \addtolength{\textwidth}{#1\oddsidemargin}
  \addtolength{\textwidth}{#1\evensidemargin}
  \addtolength{\topmargin}{-#1\topmargin}
  \addtolength{\oddsidemargin}{-#1\oddsidemargin}
  \addtolength{\evensidemargin}{-#1\evensidemargin}
  }

\shrinkmargins{.7}

\theoremstyle{plain}

\newtheorem{theorem}{Theorem}[section]
\newtheorem{cor}[theorem]{Corollary}
\newtheorem{lemma}[theorem]{Lemma}
\newtheorem{proposition}[theorem]{Proposition}
\newtheorem{question}[theorem]{Question}

\newtheorem{defi}[theorem]{Definition}
\newtheorem{ex}[theorem]{Example}

\theoremstyle{remark}
\newtheorem{rem}[theorem]{Remark}
\numberwithin{equation}{section}

\def \Z { \mathbb{Z}}
\def \Q { \mathbb{Q}}

\def \dim { \text{dim}}
\def \gal { \text{Gal}}

\def \det { \text{det}}

\def \sp { \text{Span}}
\def \Disc { \text{Disc}}
\def \tr { \text{tr}}
\def \Cl { \text{Cl}}
\def \GL { \text{GL}}
\def \SL { \text{SL}}
\def \Sym { \text{Sym}}
\def \rank { \text{rank}}

\begin{document}
\thispagestyle{empty}
\setcounter{tocdepth}{1}

\title{Integral trace forms associated to cubic extensions}
\author{Guillermo Mantilla-Soler}


\date{}

\begin{abstract}
Given a nonzero integer $d$, we know by Hermite's Theorem that
there exist only finitely many cubic number fields of discriminant
$d$. However, it can happen that two non-isomorphic cubic fields
have the same discriminant.  It is thus natural to ask whether
there are natural refinements of the discriminant which completely
determine the isomorphism class of the cubic field. Here we
consider the trace form  $q_K: \tr_{K/
\mathbb{Q}}(x^2)|_{O^{0}_{K}}$ as such a refinement. For a cubic
field of fundamental discriminant $d$ we show the existence of an
element $T_K$ in Bhargava's class group
$\Cl(\Z^{2}\otimes\Z^{2}\otimes\Z^{2}; -3d)$ such that $q_K$ is
completely determined by $T_K$. By using one of Bhargava's
composition laws, we show that $q_K$ is a complete invariant whenever $K$ is totally real and of fundamental discriminant.\\
\end{abstract}

\maketitle

\tableofcontents

\section{Introduction}

\subsection{Generalities}
A difference between quadratic and non-quadratic number fields is
that in the former case, the fields are totally characterized by
their discriminant. One natural choice for a ``refined
discriminant" is given by the isometry class with respect to the
trace form of the lattice defined by the maximal order. The
purpose of this paper is to give a detailed analysis of this
refinement for cubic extensions, and to show under which
conditions this refinement characterizes the field. Given a number
field $K$ with maximal order $O_K$ we consider the trace form
$\tr_{K/\Q}(x^2)|_{O_{K}}$. A natural question is:

\begin{question}\label{Q}
Do there exist two non-isomorphic number fields $K$ and $L$ such
that their corresponding trace forms are isomorphic?
\end{question}

In this paper we analyze Question \ref{Q} in the case of
cubic extensions.

\begin{defi}\label{deftracezero}
Let $K$ be a number field and let $O_K$ be its maximal order. The
\textit{trace zero module} $O_{K}^{0}$ is the set $\{x \in O_K :
\emph{\tr} _{K/\Q}(x)=0\}$.
\end{defi}
Our main result is the following: \\

\noindent \textbf{Theorem A} (Theorem \ref{principal} below).
\textit{Let $K$ be a cubic number field of positive, fundamental
discriminant. Let $L$ be a number field such that there exists an
isomorphism of quadratic modules}
\begin{eqnarray*}
  \langle O^{0}_{K} , \tr_{K/ \mathbb{Q}}(x^2)|_{O^{0}_{K}} \rangle & \cong & \langle
O^{0}_{L},\tr_{L/ \mathbb{Q}}(x^2)|_{O^{0}_{L}} \rangle,
\end{eqnarray*}
\textit{and assume $9 \nmid d_{L}$. Then $K \cong L$.}\\

\subsection{Outline of the paper}
We start by analyzing Question \ref{Q} for general cubic fields.
For this purpose we consider first the case in which the common
discriminant of $K$ and $L$ is not fundamental.\footnote{$d$ is a
fundamental discriminant if it is the discriminant of a quadratic
field.}

\subsubsection{Non-fundamental discriminants}

In this case, we find that our proposed refinement does not
characterize the field. In other words, for non-fundamental
discriminants we have an affirmative answer to Question \ref{Q}.
We divide the class of non-fundamental discriminants into two
groups according to sign. Among the positive discriminants, we
divide them again into groups according to those that are perfect
squares, and those that are not. For each one of these cases we
show that there are some non-fundamental discriminants such that
\ref{Q} has an affirmative answer.

\begin{itemize}

\item[i)] (Negative non-fundamental discriminants) We define a
sequence of positive integers $\Sigma$ and a family of triples
$\{K_{m},L_{m}, E_{m} \}_{m \in \Sigma}$, with the following
properties (see Lemma \ref{elliptic}):

\begin{itemize}
\item[$\bullet$] $K_{m},L_{m}$ are two non-isomorphic cubic fields
with discriminant $-3n^{2}$, where $n$ is a positive integer
depending only on $m$.

\item[$\bullet$] An elliptic curve $E_{m}$ defined over $\Q$ such
that $E_{m}[3](\Q)$ determines completely a ternary quadratic form
equivalent to both $\tr_{K/\Q}(x^2)|_{O_{K_{m}}}$ and
$\tr_{K/\Q}(x^2)|_{O_{L_{m}}}$.

\end{itemize}

\item[ii)] (Square discriminants) In this case we prove (see
Theorem \ref{galois}) a generalization  of a result of Conner and
Perlis (\cite{conner}, Theorem IV.1.1 with $p=3$). Let $K$ and $L$
be two Galois cubic number fields of the same discriminant and let
$M$ be either $O_{K}$ or $O_{K}^{0}$. Then $\tr_{K/\Q}(x^2)|_{M}$
and $\tr_{K/\Q}(x^2)|_{M}$ are equivalent. Since there are
examples of non-isomorphic Galois cubic fields of the same
discriminant, Question \ref{Q} has a positive answer for such
cases.

\item[iii)](Positive, non-fundamental, non-square discriminants)
See example \ref{examplePositiveNonfundamental} for two fields
with positive, non-squarefree, non perfect square discriminant and
isometric integral trace forms.
\end{itemize}

\subsubsection{Main results}

For fields of fundamental discriminant we see, thanks to Lemma
\ref{trace0refinement}, that the binary quadratic form
$\tr_{K/\Q}(x^2)|_{O^{0}_{K}}$ is a refinement of the
discriminant. Hence, we reformulate Question \ref{Q}.

\begin{question}\label{Q2}
Do there exist two non-isomorphic cubic fields $K$ and $L$ such
that the forms $\emph{\tr}_{K/\Q}(x^2)|_{O^{0}_{K}}$ and
$\emph{\tr}_{L/\Q}(x^2)|_{O^{0}_{L}}$ are isomorphic?
\end{question}

Although Question \ref{Q2} has relevance for us only for
fundamental discriminants, we note that the examples i), ii) and
iii) described above also answer \ref{Q2} in an affirmative way.
On the other hand, for fundamental discriminants (see Figure $1$),
class field theory provides examples of non isomorphic cubic
fields of the same discriminant. Among the fields with negative
discriminants we found examples giving an affirmative answer to
Question \ref{Q2}.

It is clear thanks to the results developed so far, that one
should consider working over cubic fields of fundamental
discriminant. We show for such discriminants that the trace form
is equal, as an element of a narrow class group, to the Hessian
multiplied by an element that only depends on the discriminant.\\

\noindent \textbf{Theorem B} (Theorem \ref{grouprel} below).
\textit{Let $K$ be a cubic field with discriminant $d_K$. Assume
that $d_K$ is fundamental and that $3 \nmid d_K$. Let
$F_{K}=(a,b,c,d)$ be a cubic in the $\emph{\GL}
_{2}(\Z)$-equivalence class defined by $K$. Then
$\frac{1}{2}q_{K}*C_{d_K} =H_{K}^{\pm 1}$ as elements of
$\emph{\Cl}^{+}_{\Q(\sqrt{-3 d_K})}$, where $C_{d_K}
=(3,0,\frac{d_{K}}{4})$ or $C_{d_K} =(3,3,\frac{d_{K}+3}{4})$ in
accordance with whether $d_k \equiv 0 \pmod 4$ or $d_k \equiv 1 \pmod 4$.}\\

By reformulating all of this in the language of Bhargava's
composition of cubes (see \cite{bhargava}), we show that the trace
form arises naturally as a projection of a cube determined by the
field.\\

\noindent \textbf{Theorem C} (Theorem \ref{grouprel2} below).
\textit{Let $K$ be a cubic field with discriminant $d_K$ and
associated cubic form $F_K=(a,b,c,d)$. Assume that $d_K$ is
fundamental and that $3$ does not ramify. Then there exists
$T_{F_K} \in \emph{\Cl}(\Z^{2}\otimes\Z^{2}\otimes\Z^{2}; -3d_K)$
such that $(\pi_{1} \circ \phi)(T_{F_K}))^{\pm1}=\frac{1}{2}q_{K}$
as elements of
$\emph{\Cl}^{+}_{\Q(\sqrt{-3d_K})}$}.\\

In this setting, Theorem A follows from Theorem \ref{morales}
which is the modern version of a theorem of Eisenstein
(see\cite{E}). By reformulating Theorem A, see Theorem
\ref{ultimo} and its corollary, we obtain one inequality of the
classical Scholz reflection principle (see \cite{Scholz}).

We remark that Theorem A can be obtained with the tools developed
by Eisenstein in \cite{E}. However, we have decided to use
Bhargava's theory of $2\times 2 \times 2$ orbits of cubes, to
suggest that it might be possible to use some other prehomogeneous
spaces to ``generalize " Theorem A to higher dimensions.

\newpage

\section{Basic facts}

\begin{defi}
Let $G$ be a free abelian group. We say that a map $$q : G
\rightarrow \Z$$ is a \textit{quadratic form} if :

\begin{itemize}

\item $q(nx)=n^2q(x)$ for all integer $n$,

\item The map $B_q: G \times G \rightarrow \frac{1}{2}\Z$ defined
as $B_q(x,y) = \frac{1}{2}(q(x+y)-q(x)-q(y))$ is $\Z$-bilinear.
\end{itemize}

\end{defi}
\begin{rem}

Let $\langle G, q \rangle$ be a quadratic $\Z$-module of rank $n =
\rank(G)$. After choosing a basis, we can think of $q$ as a
homogeneous polynomial in $n$ variables of degree two, i.e $q \in
(\Sym^{2}\Z^{n})^{*}$.  There is a natural action of $\GL_2(\Z)$
on $(\Sym^{2}\Z^{n})^{*}$. Under this action, $q$ and $q_1$ belong
to the same orbit if and only if $ \langle G, q \rangle$ is
isometric to $\langle G_1, q_1 \rangle$. Abusing notation we will
denote this by $q \sim_{\GL_{2}(\Z)} q_1$.

\end{rem}

Let $K$ be a number field and let $O_K$ be its maximal order. The
map
\begin{displaymath}
\begin{array}{cccc}
\tilde{q}_K : & O_K  &  \rightarrow & \Z  \\  & x & \mapsto &
\tr_{K/\Q}(x^2)
\end{array}
\end{displaymath}
defines a quadratic form with corresponding bilinear form
$$B_K(x,y ) = \tr_{K/\Q}(xy)|_{O_K}.$$

Thus, we have that $ \langle O_K,\tilde{q}_K \rangle$ is a
quadratic $\Z$-module and its discriminant is precisely the
discriminant of $K$. Thus, if $K$ and $L$ are two number fields
such that $ \langle O_K,\tilde{q}_K \rangle$ and $ \langle
O_L,\tilde{q}_L \rangle$ are isomorphic quadratic $\Z$-modules,
then we have

\begin{itemize}

\item $[K : \Q] = [L : \Q],$

\item $\Disc(K)  = \Disc(L).$

\end{itemize}

Therefore the isomorphism class of $\langle O_K,\tilde{q}_K
\rangle$ is to us a natural refinement of the discriminant.

\begin{lemma}\label{tamanio}
Let $K$ be a number field of degree $n$ and let $G_K = \Z +
O_{K}^{0}.$ We have
\begin{displaymath}
|O_K/G_K| = |\emph{\tr}_{K/\Q}(O_K) / nZ|.
\end{displaymath}
\end{lemma}

\begin{cor}\label{invariante}
Let $K$ and $L$ be number fields. If
$$ f  :  \langle O_K,B_K \rangle \quad \rightarrow \quad  \langle O_L,B_L
\rangle$$ is an isomorphism, then $\emph{\Disc}(G_K) =
\emph{\Disc}(G_L).$

\end{cor}

\begin{proof}

 Since $\tr_{L/\Q}(f(x)f(y)) = \tr_{K/\Q}(xy)$ for all $x,y \in O_K$ we have that $\tr_{K/\Q} :
O_K \twoheadrightarrow \Z$ implies $\tr_{L/\Q} : O_L
\twoheadrightarrow \Z$. Since $f$ is an isometry, the argument is
symmetric in $K$ and $L$. By Lemma \ref{tamanio} we have
$|O_K/G_K|= |O_L/G_L|$. Hence
$$\Disc(G_K) = |O_K/G_K|^2\Disc(O_K)= |O_L/G_L|^2\Disc(O_L)=
\Disc(G_L).$$
\end{proof}

For a number field  $K$, we denote $q_K =
\tilde{q}_{K}|_{O_{K}^{0}}$.

\begin{lemma}\label{trace0refinement}

Let $K$, $L$ be two  number fields of degree $n$. Assume that $K$
and $L$  both have discriminants that are squarefree at all primes
dividing $n$. Further, suppose that $\langle O^{0}_K, q_K \rangle$
and $\langle O^{0}_L, q_L \rangle$ are isomorphic. Then $K$ and
$L$ have the same discriminant.

\end{lemma}
\begin{proof}
Since $\Disc(G_{K})=\Disc(G_{L})$, we have that
$$|O_K/G_K|^2\Disc(O_K)=|O_L/G_L|^2\Disc(O_L).$$ The result now
follows from Lemma \ref{tamanio}.
\end{proof}

\begin{proposition}\label{p-ramifGaloisCase}

Let $K$ be a Galois number field of prime degree $p$. Then $p$
ramifies in $K$ if and only if $\emph{\tr}_{K/\Q}(O_K)=p\Z$.

\end{proposition}

\begin{proof}
It is clear that $\tr_{K/\Q}(O_K)=p\Z$ implies that $p$ ramifies in
$K$. Next, assuming that $p$ ramifies, let $P$ be the unique prime
of $O_K$ lying above $p$. By hypothesis we have that $|O_K/P|=p$.
In particular $P$ is a maximal $\Z$-submodule of $O_K$. Since $1
\notin P$, we must have that $O_K=\Z+P$. Since $P$ is Galois
invariant, $\tr_{K/\Q}(P) \subseteq P \cap \Z =p\Z$. Thus
$\tr_{K/\Q}(O_K) =\tr_{K/\Q}(\Z+P) \subseteq p\Z$.
\end{proof}

\section{Galois fields and rational $3$-torsion}
In this section we explain some situations in which Questions
\ref{Q} and \ref{Q2} have positive answers. The examples in this section are
characterized by having discriminants with a nontrivial square
factor.

The following result is a generalization of a theorem of Conner
and Perlis (\cite{conner}, Theorem IV.1.1) for $p=3$.
\begin{theorem}\label{galois}
Let $K$ and $L$ be two Galois, cubic number fields of discriminant
$D=d^2$. We have $$\langle O_{K}^{0},q_K \rangle \cong \langle
O_{L}^{0}, q_L \rangle \cong \begin{cases} 2d(x^2+xy+y^2) &
\text{if $3\nmid d$} \\ \frac{2d}{3}(x^2+xy+y^2) & otherwise.
\end{cases}$$ Moreover, there exists such an isometry that extends to one between  $\langle O_K,\tilde{q}_K
\rangle$ and $\langle O_L,\tilde{q}_L \rangle$.
\end{theorem}

\begin{proof}
Assume first that $3 \nmid D$. By Hilbert's $132$ (see
\cite{hilbert}) we have $O_K = e_{1}\Z \oplus e_{2}\Z \oplus
e_{3}\Z$, where $\sigma(e_1)=e_2$, $\sigma(e_2)=e_3$, and $\sigma$
is a generator of Gal$(K/\Q).$ Since $3$ does not ramify,
Proposition \ref{p-ramifGaloisCase} implies that
$\tr_{F/\Q}(e_{1})=1$, and furthermore, that $O_{K}^{0}=
(e_{1}-e_{2})\Z \oplus(e_{1}- e_{3})\Z.$ Let
$a=\tr_{F/\Q}(e_{1}^{2})$ and $b=\tr_{F/\Q}(e_{1}e_{2})$.\\
Then $M =\left(
\begin{array}{ccc}
  a & a-b &a-b\\
a-b & 2a-2b &a-b\\
a-b & a-b &2a-2b
\end{array}
\right)$ (respectively $M_0 =(a-b)\left(
\begin{array}{cc}
  2 & 1\\
  1 & 2
\end{array}
\right)$)\\ represents the trace form over $O_K$ in the basis
$\{e_1,e_1-e_2,e_1-e_3\}$ (respectively the trace form over
$O_{K}^{0}$ in the basis $\{e_1- e_2,e_2-e_3\}$). Note that $a+2b
=(\tr_{F/\Q}(e_{1}))^{2}=1$, thus $D
=\det(M)=(a-b)^{2}(a+2b)=(a-b)^{2}$. By the Cauchy-Schwartz
inequality, $a-b>0$, hence $d=a-b$, which implies that
$a=\frac{1+2d}{3}$ and $b=\frac{1-d}{3}$. Thus, every cubic field
of discriminant $d^2$, with $3\nmid d$, has an integral basis for
which the trace form over $O_K$ has representative matrix $M$
(respectively trace form over $O_{K}^{0}$ has representative
matrix $M_0$).\\ On the other hand, if $3 \mid d$, Proposition
\ref{p-ramifGaloisCase} and Lemma \ref{tamanio} imply that
$O_{K}=\Z \oplus O_{K}^{0}$. Hence, $\tilde{q}_K$ is totally
determined by $q_{K}=\tilde{q}_{K}|_{O_{K}^{0}}$. Since every
integral quadratic form of discriminant $-3$ is
$\SL_{2}(\Z)$-equivalent to $(x^2+xy+y^2)$, the result follows
from the following claim.
\end{proof}

\noindent \textbf{Claim:} $\frac{3}{2d}q_{K}$ is an integral,
primitive, binary quadratic form of discriminant $-3$.

\begin{proof}[Proof of claim:]
Let $\{ \alpha,\beta \}$ an integral basis for $O_{K}^{0}$. Let
$O_{\alpha} \subseteq O_{K}^{0}$ be the $\Z$-module generated by
$\{\alpha, \sigma(\alpha)\}$, where $\sigma$ is a generator for
$\gal(K/\Q)$. Since $\alpha \notin \Z$, we know that $\alpha$ and
$\sigma(\alpha)$ are distinct elements of $O_K$ with the same
norm. In particular, $\sigma(\alpha)$ cannot be a rational
multiple of $\alpha$, so $\rank_\Z(O) = 2.$ Thus,
$[O_{K}^{0}:O_{\alpha}]$ is finite, and moreover
$\sigma(\alpha)=m\alpha +[O_{K}^{0}:O_{\alpha}]\beta$ for some
integer $m$. Note that
$(\tr_{K/\Q}(\alpha^{2}),2\tr_{K/\Q}(\alpha\beta),\tr_{K/\Q}(\beta^{2} ))$
and
$(\tr_{K/\Q}(\alpha^{2}),2\tr_{K/\Q}(\alpha\sigma(\alpha)),\tr_{K/\Q}(\sigma(\alpha)^{2}
))$ represent $q_{K}$ in the bases $\{\alpha, \beta\}$ and
$\{\alpha, \sigma(\alpha)\}$ respectively.  Hence
\begin{equation}\label{ecuacionestrazas} \tr_{K/\Q}(\alpha^{2})\tr_{K/\Q}(\sigma(\alpha)^{2}
)-\tr_{K/\Q}^{2}(\alpha\sigma(\alpha))=[O_{K}^{0}:O_{\alpha}]^{2}(\tr_{K/\Q}(\alpha^{2})\tr_{K/\Q}(\beta^{2}
)-\tr_{K/\Q}(\alpha\beta)).\end{equation} Since $\Disc(K)=d^2$ and
$O_{K}=\Z+O^{0}_{K}$,
$\frac{d^2}{3}=\tr_{K/\Q}(\alpha^{2})\tr_{K/\Q}(\beta^{2}
)-\tr_{K/\Q}(\alpha\beta)$. On the other hand since $\alpha \in
O_{K}^{0}$, $\tr_{K/\Q}(\alpha^{2})=-2\tr_{K/\Q}(\alpha\sigma(\alpha))$,
and the left hand side of (\ref{ecuacionestrazas}) is
$3\tr_{K/\Q}^{2}(\alpha\sigma(\alpha))$. Thus,
\begin{equation}
\tr_{K/\Q}(\alpha\sigma(\alpha))=\pm[O_{K}^{0}:O_{\alpha}]\frac{d}{3}.
\end{equation}
In particular we see that $\frac{d}{3}$ divides
$\frac{1}{2}\tr_{K/\Q}(\alpha^{2})$. Exchanging the roles of $\alpha$
and $\beta$ we see that $\frac{d}{3}$ also divides
$\frac{1}{2}\tr_{K/\Q}(\beta^{2})$. Now consider
$\sigma(\alpha)=m\alpha +[O_{K}^{0}:O_{\alpha}]\beta$. Multiplying
both sides by $\alpha$ and then taking traces we see that
$\frac{d}{3}$ divides $\tr_{K/\Q}(\alpha\beta)$. We conclude that
$(\tr_{K/\Q}(\alpha^{2}),2\tr_{K/\Q}(\alpha\beta),\tr_{K/\Q}(\beta^{2} ))$
can be written as $\frac{2d}{3}f$, with $f$ an integral quadratic
form of discriminant $-3$.
\end{proof}

\begin{ex}
Let $K$ and $L$ be cubic fields defined by $x^3 +6x^2 -9x +1$ and
$2x^3 + 3x^2 -9x + 2$ respectively. One sees that $K$ and $L$ are
non-isomorphic fields of discriminant $3969$ by direct
computation,  for instance  \emph{\text{regulator}}$(K) \neq$
\emph{\text{regulator}}$(L)$.
\end{ex}

We conclude that the trace form does not characterize the field in
the case that the discriminant is a square. Proposition
\ref{elliptic} below is an indication that the case of square
discriminant is not the only case that should be reconsidered.
Namely, one should also consider the non-squarefree case. Cubic
fields of a fixed discriminant $\Delta$ can be parametrized by a
subset of rational points on a certain elliptic curve. Assume that
$L=\Q(\beta)$ is a cubic field defined by the equation $x^3+px+q
\in \Z[x]$. If $O_{L}=\Z[\beta]$, then $\Disc(L)=-27q^{2}-4p^3$.
Hence if $K$ is a cubic field of discriminant $\Delta$, one could
try to find a cubic field $L$ of the same discriminant by finding
rational points $(-\frac{p}{3}, \pm\frac{q}{2})$ of
$y^{2}=x^{3}-\frac{\Delta}{108}.$ Using this idea, we construct a
family of non isomorphic cubic fields with prescribed
discriminant. We need the following result from algebraic number
theory (see \cite{marcus}).
\begin{proposition}\label{simplecubicbasis} Let $m$ be a non
perfect cube integer and $\alpha$ a root of $x^{3}-m$. Write $m
=m_{f}m^{2}_{s}$ with $m_{f}$ squarefree and
\emph{gcd}$(m_{f},m_{s})=1$. Suppose that $m \not\equiv \pm 1\pmod
9$. Then $\{1, \alpha, \alpha^{2}/m_{s}\}$ is an integral basis
for $K_{m}=\Q(\alpha)$; in particular \emph{\Disc}$(K_{m})=
-27(m_{s}m_{f})^{2}$.

\end{proposition}

Let $\Sigma = \{m \in \mathbb{N} \setminus \mathbb{N}^{3}|
m_{s}\neq 1, m_{f}m_{s} \not\equiv \pm 1\pmod 9, m \not\equiv \pm
1 \pmod 9 \}.$
\begin{proposition}\label{elliptic}
Let $m \in \Sigma$ and $K_{m}, L_{m}$ be the cubic fields defined
by $x^{3}-m$ and $x^{3}-m_{f}m_{s}$ respectively, with $m_f, m_s$
as in Proposition \ref{simplecubicbasis}. Then $K_{m}, L_{m}$ are
cubic fields with equivalent trace forms, and have discriminant
$-3(3m_{f}m_{s})^{2}$.
\end{proposition}

\begin{proof}
By the discussion above and Proposition \ref{simplecubicbasis} we
have that $K_{m}$ defines the rational elliptic curve $E_{m} :
y^{2} =x^3 + \frac{m_{f}^{2}m_{s}^{2}}{4}$. A simple calculation
shows that $E_{m}[3](\Q)=\{\infty,
(0,\frac{m_{f}m_{s}}{2}),(0,-\frac{m_{f}m_{s}}{2})\}$, and these
points define the field $L_{m}$. Let $P$ be a generator of
$E_{m}[3](\Q)$ and $M_{m}=\left(
\begin{array}{ccc}
  3 & 0 &0\\
0 & 0 &6y(p)\\
0 & 6y(p) & 0
\end{array}
\right)$. Then $M_{m}$ represents simultaneously the trace form in
$O_{K_{m}}$ and $O_{L_{m}}$ with respect to the bases given by
Proposition \ref{simplecubicbasis}.
\end{proof}

The pair of number fields given by Proposition \ref{elliptic} need
not to be isomorphic, as the following example demonstrates.

\begin{ex}
Let $m=12$ so that $K_{12}$ and $L_{12}$ are the cubic fields
defined by $x^3-12$ and $x^3-6$ respectively. Then $\langle
O_K,\tilde{q}_{K_{12}} \rangle$ and $ \langle
O_L,\tilde{q}_{L_{12}} \rangle$ are isomorphic to $\langle
\Z^{3},3x^{2}+36yz \rangle$. One sees that $K_{12}$ and
$L_{12}$ are non-isomorphic fields of discriminant $-2^{2}3^{5}$ by direct computation, for instance $7$ splits in $L_{12}$ but it is inert in
$K_{12}$.
\end{ex}

Recall that for Galois cubic fields of fixed discriminant there is
only one possibility for the trace form (see Theorem
\ref{galois}). This follows, since after a suitable scaling we are
left with a binary quadratic form of discriminant $-3$. Inspired
by this, we began looking for discriminants $D$ of totally real
cubic fields satisfying the following conditions:

\begin{itemize}

\item[(i)] $D$ is a non-perfect square.

\item[(ii)] $D$ is non-fundamental.

\item[(iii)] Up to squares factors and factors of $3$, $-D$
defines an imaginary quadratic field of class number $1$.

\item[(iv)] There are at least two cubic fields of discriminant
$D$.

\end{itemize}

It turns out that the first $D$ satisfying the above conditions
(see tables at the end of \cite{ET}), is $D=66825=3^{5}5^{2}11$.
For this value of $D$ we have:

\begin{ex}\label{examplePositiveNonfundamental}
Let $K$ and $L$ be the cubic fields defined by $2x^3+3x^2-21x+4$
and $x^3+9x^2-18x-3$ respectively. Then $\langle O_K,\tilde{q}_K
\rangle$ and $\langle O_L,\tilde{q}_L \rangle$ are isomorphic to
$\langle \Z^{3},3x^{2}+90(y^2+yz+3z^2) \rangle$. One sees that $K$
and $L$ are non-isomorphic fields of discriminant $3^{5}5^{2}11$
by direct computation,  for instance \emph{regulator}$(K) \neq$
\emph{regulator}$(L)$.
\end{ex}

None of our results so far yield positive answers to Questions
\ref{Q} or \ref{Q2} with fundamental discriminant.  It is thus
natural to ask whether those questions have negative answers in
the special case where the discriminant of the cubic field is
fundamental. Moreover, under this circumstances we will exhibit a
more convenient refinement. To describe this, let $K$ be a cubic
number field and recall our notation $q_k =
\tilde{q}_{K}|_{O_{K}^{0}}$. Then $q_K$ is an integral, binary
quadratic form. Moreover under the fundamental discriminant
hypothesis, the isometry class of $ \langle O^{0}_K, q_K \rangle$
is a refinement of the discriminant, as shown in Lemma
\ref{trace0refinement}.

\section{Cubic fields with fundamental discriminant}
Throughout section $4$, all cubic fields are assumed to have
fundamental discriminant. The first question that comes to mind is
the following: for which fundamental discriminants $d$ does there
exist a cubic field with discriminant $d$? Moreover, we would like
to know for which values of $d$ there is more than one isomorphism
class of cubic fields of discriminant $d$. It turns out that class
field theory gives nice answers to these questions. Let $K$ be a
cubic field of fundamental discriminant $d$ and  Galois closure
$\widetilde{K}$. Clearly, $\Q(\sqrt{d}) \subseteq \widetilde{K}$,
and moreover, this extension is unramified. Since $d$ is a
fundamental discriminant, $\gal(\widetilde{K}/\Q) \cong S_{3}$.
Hence $[\widetilde{K}:\Q(\sqrt{d})]=3$, and
$\widetilde{K}/\Q(\sqrt{d})$ is abelian. Therefore, if $H_{d}$
denotes the Hilbert class field of $\Q(\sqrt{d})$, and
$\Cl_{\Q(\sqrt{d})}$ denotes the ideal class group of
$\Q(\sqrt{d})$), we have the following diagram:
$$
\xymatrix{  & H_{d} \ar@{-}[d]^{H_{K}}\ar@/^.8cm/@{-}[ddr]^{\Cl_{\Q(\sqrt{d})}}  &  \\
     & \widetilde{K}\ar@{-}[dd]^{S_{3}}\ar@{-}[ld]_{\Z/2\Z}\ar@{-}[rd]^{\Z/3\Z}   & \\
             K\ar@{-}[rd] & &  \Q(\sqrt{d})\ar@{-}[ld]^{\Z/2\Z} \\
            & \Q &  \\
            &  \mathbf{Figure \quad 1}  & }$$
Thus, if we start with $K$ as above, we obtain $H_{K}$, an index
three subgroup of $\Cl_{\Q(\sqrt{d})}$. Conversely, it can be
shown (see \cite{hasse}) that the fixed field of an index $3$
subgroup of $\Cl_{\Q(\sqrt{d})}$ corresponds to the Galois closure
of a cubic field of discriminant $d$. Hence we have the following
proposition:

\begin{proposition}[Hasse, \cite{hasse}]
The number of isomorphism classes of cubic fields of discriminant
$d$ is $(3^{\emph{r}_{3}(d)}-1)/2$, where
$\emph{r}_{3}(d)=\emph{\dim}_{\mathbb{F}_{3}}(\emph{\Cl}_{\Q(\sqrt{d})}
\otimes_{\Z} \mathbb{F}_3)$.

\end{proposition}

\begin{cor}[Hasse, \cite{hasse}]
There exists a cubic field $K$ of discriminant $d$ if and only if
$\emph{Cl}_{\Q(\sqrt{d})}[3] \neq 0$.
\end{cor}

Section $3$ has given affirmative answers to Questions \ref{Q} and
\ref{Q2} for non-fundamental discriminants. The following example
shows us that among fundamental discriminants one still finds
positive answers to Questions \ref{Q} and \ref{Q2}.

\begin{ex}\label{primernegativo}
The fundamental discriminant of least absolute value with
$\emph{r}_{3}(d)>1$ is $d=-3299$. For this value of $d$,
$\emph{\Cl}_{\Q(\sqrt{d})} \cong \Z/3\Z \oplus \Z/9\Z$; hence
there exist four non isomorphic cubic fields of discriminant
$-3299$. Among these four fields, the ones defined by $x^3+2x+11$
and $x^3-16x+27$ have isometric trace zero parts.
\end{ex}
Cubic fields with squarefree discriminants lead us to $3$-torsion
of class groups of quadratic fields. There is another very well
known source of class groups of quadratic fields, namely binary
quadratic forms. Let us recall briefly how these two are
connected. Let $\Delta$ be a non perfect square integer and let
$\Gamma_{\Delta}$ (respectively $\Gamma^{1}_{\Delta}$) be the set
of $\GL_{2}(\Z)$-equivalence classes (respectively
$\SL_{2}(\Z)$-equivalence classes) of primitive, binary quadratic
forms of discriminant $\Delta$. Gauss composition gives a group
structure to $\Gamma^{1}_{\Delta}$, and furthermore this group is
isomorphic to \textit{the narrow class group}
$\Cl^{+}_{\Q(\sqrt{\Delta})}$. In particular $|\Gamma_{\Delta}|
\leq |\Cl^{+}_{\Q(\sqrt{\Delta})}|$. Now, let $K$ be a cubic field
of discriminant $d$ not divisible by $3$. Thanks to the next
lemma, the $\GL_{2}(\Z)$-equivalence class of $[\frac{1}{\tiny 2
}q_{K}]$ defines an element of $\Gamma_{-3d}$. Thus, if we denote
by $\mathcal{C}_{d}$ the set of isomorphism classes of cubic
fields of discriminant $d$, we have the following map:
\begin{eqnarray*}
\Phi_{d} : \mathcal{C}_{d} & \longrightarrow & \Gamma_{-3d} \\
     K & \mapsto & [\frac{1}{\tiny 2
}q_{K}].
\end{eqnarray*}
Since $\Cl^{+}_{\Q(\sqrt{9897})} \cong \Z/3\Z$ and
$|\mathcal{C}_{-3299}|=4$, the previous example can be restated as
the non-injectivity of $\Phi_{-3299}$.

\begin{lemma}\label{integralofdisc-3d}
Let $K$ be a cubic field with fundamental discriminant $d$. Then $
\frac{1}{2}q_{K}$ is an integral, binary quadratic form of
discriminant $-3d$.
\end{lemma}

\begin{proof}
Note that $\Disc(q_{K})=-4\Disc(O^{0}_{K})=$
$\frac{-4|O_K/G_K|^{2}d}{3}$. Since $d$ is fundamental, $9 \nmid
d$. In particular, $\tr_{K/\Q}$ is a surjection from  $O_K$ to
$\Z$ and thanks to Lemma \ref{tamanio} we have
$\Disc(q_{K})=-12d$. Note that if $x \in O^{0}_{K}$, then
$\tr_{K/\Q}(x^2)=\tr_{K/\Q}(x^2)-\tr_{K/\Q}^2(x) \in 2\Z$, hence $ \frac{1}{2}q_{K}$ is
integral.
\end{proof}

\begin{rem}
In fact, if $3\nmid d$, $\frac{1}{2}q_{K}$ is primitive as seen in
Corollary \ref{notdivby3}.
\end{rem}

Often it is more convenient to work with primitive forms rather
than general ones. Since $q_K \sim_{\GL_{2}(\Z)} q_L$ if and only
if  $aq_{K} \sim_{\GL_{2}(\Z)} aq_{L}$ for any non zero rational
number $a$, the previous remark will allow us to restrict
ourselves to primitive forms.

\section{Trace form and class groups }

In this section we calculate $q_{K}$ explicitly, and then show
that for positive fundamental discriminants, $q_K$ characterizes
the field. To this end, we start by recalling the theorem of
Delone-Faddeev-Gan-Gross-Savin on parametrization of cubic rings
(see \cite{Delone-Faddeev}, \cite{Gan-Gross} or \cite{belabas}).
Every conjugacy class of a cubic ring $R$, has associated to it a
unique integral binary cubic form
$(a,b,c,d):=F(x,y)=ax^3+bx^2y+cxy^2+dy^3$ up to
GL$_{2}(\Z)$-equivalence. Let $K$ be a cubic number field and $F$
the form associated to its maximal order. Among the properties of
$F$ we have the following:
\begin{itemize}
\item $K=\Q(\theta)$, where $\theta \in K$ is a root of
$F_{K}(x,1).$

\item
$d_{K}:=\Disc(K)=\Disc(a,b,c,d)=b^2c^2-27a^2d^2+18abcd-4ac^3-4b^3d.$

\item The Hessian form of $F$, $H_{F}= (P,Q,R):= Px^2+Qxy+Ry^2$,
has discriminant $-3d_{K}$, where $$P=b^2-3ac, Q=bc-9ad,
R=c^2-3bd.$$

\item $H_F$ is covariant with respect to the GL$_{2}(\Z)$-action
on binary cubic forms and on binary quadratic forms.

\item $\mathcal{B}=\{1,-a\theta,\frac{d}{\theta}\}$ is a
$\Z$-basis of $O_{K}$.

\item If $d_{K}$ is fundamental, then $H_{F}$ is a primitive, binary
quadratic form.
\end{itemize}

\begin{lemma}\label{hess}
Let $\alpha= -a\theta$ and $\beta=\frac{d}{\theta}$. Then $H_{F}$
is realized as the integral quadratic form
$\frac{3}{2}\emph{\tr}_{K/\Q}(X^2)$ over the $\Z$-module
$O_{K}^{\mathcal{B}}=\emph{\sp}_{\Z}\{\alpha-\frac{\emph{\tr}_{K/\Q}(\alpha)}{3},
\beta-\frac{\emph{\tr}_{K/\Q}(\beta)}{3}\}$.
\end{lemma}

\begin{proof}

Note that  $a^2F(\frac{x}{a},1)$ and $d^2F(1,\frac{x}{d})$ are the
minimal polynomials over  $\Q$ of $\alpha$ and $\beta$
respectively. Hence, $\tr_{K/\Q}(\alpha)=b$,
$\tr_{K/\Q}(\beta)=-c$, $\tr_{K/\Q}(\alpha \beta)=-3ad$,
$\tr_{K/\Q}(\alpha^2)=b^2-2ac$, and $\tr_{K/\Q}(\beta^2)=c^2-2bd$.
From this and a simple calculation the result follows.
\end{proof}

\begin{proposition}\label{0trace}

Let $\alpha_0=\alpha-\frac{\emph{\tr}_{K/\Q}(\alpha)}{3}$ and
$\beta_0=\beta-\frac{\emph{\tr}_{K/\Q}(\beta)}{3}$.  Then

$$O_{k}^{0}=\begin{cases}
O_1=\emph{\sp}_{\Z}\{\alpha_0,3\beta_0\} & \textrm{if $b \equiv 0 \pmod{3}$}\\
O_2=\emph{\sp}_{\Z}\{3\alpha_0,\beta_0\} & \text{if $c \equiv 0 \pmod{3}$}\\
O_3=\emph{\sp}_{\Z}\{\alpha_0-\beta_0,3\beta_0\} &\text{if $ b \equiv -c \pmod{3}$}\\
O_4=\emph{\sp}_{\Z}\{\alpha_0+\beta_0,3\beta_0\} & \text{if $b
\equiv c \pmod{3}.$}
\end{cases} $$
\end{proposition}

\begin{proof}
By Lemma \ref{hess},
$(\frac{3}{2}\tr_{K/\Q}(X^2)|O_{K}^{\mathcal{B}})=-3d_{K}$ or,
equivalently,
$(\frac{1}{2}\tr_{K/\Q}(X^2)|O_{K}^{\mathcal{B}})=-\frac{1}{3}d_{K}$.
On the other hand,
\begin{equation}
-3d_{K}= (\frac{1}{2}\tr_{K/\Q}(X^2)|O_{K}^{0})
=[O_{K}^{\mathcal{B}}:O_{K}^{0}]^{2}(\frac{1}{2}\tr_{K/\Q}(X^2)|O_{K}^{\mathcal{B}}).
\end{equation}
It follows that $[O_{K}^{\mathcal{B}}:O_{K}^{0}]=3.$ Notice that
for each $i$, the given congruence conditions on $b$ and $c$ imply
that $O_i \subseteq O_{K}^{0}.$ Since
$[O_{K}^{\mathcal{B}}:O_i]=3$ for $i \in \{1,2,3,4\}$, the result
follows.
\end{proof}

\begin{cor}\label{explicit}
Let $K$ be a cubic field and let $F_K= (a,b,c,d)$ be a cubic form
associated to $K$. Let $ H_K=(P,Q,R)$ be the Hessian of $F_K$.
Then the binary quadratic form $\frac{1}{2}\emph{\tr}_{K/\Q}(X^2)$
on the lattice $O_{K}^{0}$ can be explicitly described as follows:

$$
\begin{cases}
(P/3, Q, 3R) & \textrm{if $b \equiv 0 \pmod{3}$}\\
(3P, Q, R/3) & \text{if $c \equiv 0 \pmod{3}$}\\
(3P, 2P-Q, \frac{P+R-Q}{3}) &\text{if $ b \equiv -c \pmod{3}$}\\
(3P, 2P+Q, \frac{P+Q+R}{3}) & \text{if $b \equiv c \pmod{3}.$}
\end{cases} $$

\end{cor}

\begin{proof}
By Lemma \ref{hess}, the matrix of $\frac{3}{2}\tr_{K/\Q}(X^2)$
over $O_{K}^{\mathcal{B}}$ in the basis $\{ \alpha_0, \beta_0 \}$
is given by $$ M=\left(
\begin{array}{cc}
  P & Q/2 \\
Q/2 & R
\end{array}
\right).
$$
Let $N_{1}=\left(
\begin{array}{cc}
  1 & 0 \\
  0 & 3
\end{array}
\right), N_{2}=\left(
\begin{array}{cc}
  3 & 0 \\
0 & 1
\end{array}
\right), N_{3}=\left(
\begin{array}{cc}
  1 & -1 \\
0 & 3
\end{array}
\right)$, and $N_{4}=\left(
\begin{array}{cc}
  1 & 0 \\
-1 & 3
\end{array}
\right)$.  Then the coordinates of the vector $N_{i}(\alpha_{0},
\beta_{0})^{t}$ form a basis of $O_{i}$, for $i \in \{1,2,3,4\}$.
Hence, $\frac{1}{3}N_{i}MN_{i}^{t}$ is the matrix that represents
$\frac{1}{2}\tr_{K/\Q}(X^2)$ over $O_{i}$ in such a basis. After
applying Proposition \ref{0trace}, the result follows.
\end{proof}

From now on whenever we choose a cubic form $F_K$ in the
GL$_{2}(\Z)$-class given by the field $K$,  what we mean by
$\frac{1}{2}q_K$ is the quadratic form in the coordinates given by
Corollary \ref{explicit}.

\begin{cor}\label{notdivby3}
Let $K$ be a cubic field with fundamental discriminant $d$ not
divisible by $3$. Then $ \frac{1}{2}q_{K}$ is a primitive,
integral, binary quadratic form of discriminant $-3d$.
\end{cor}

\begin{proof}

By Lemma \ref{integralofdisc-3d}, it remains only to prove that $
\frac{1}{2}q_{K}$ is primitive. Since $H_{k}$ is primitive and $9
\nmid -3d$, the result follows from Corollary \ref{explicit}.
\end{proof}

For a fixed $F_{K}$ in the GL$_{2}(\Z)$-class given by the field
$K$, we have found explicit relations between the binary quadratic
forms $\frac{1}{2}q_{K}$ and $H_{K}$. Since they have the same
discriminant, namely $-3d_{K}$, one could ask: what is their
relation as elements of the group $\Cl^{+}_{\Q(\sqrt{-3 d_K})}$? A
small objection to this question is that even though $H_{K}$
represents a valid element of this group, $\frac{1}{2}q_{K}$ need
not, since it may not be primitive. Yet, as Corollary
\ref{notdivby3} shows, $\frac{1}{2}q_{K}$ is primitive whenever
$3$ does not ramify in $K$. In this setting we are able to find
the following connection between forms.

\begin{theorem}\label{grouprel}
Let $K$ be a cubic field with discriminant $d_K$. Assume that
$d_K$ is fundamental and that $3 \nmid d_K$. Let $F_{K}=(a,b,c,d)$
be a cubic in the $\emph{\GL}_{2}(\Z)$-equivalence class defined
by $K$. Then $\frac{1}{2}q_{K}*C_{d_K} =H_{K}^{\pm 1}$ as elements
of $\emph{\Cl}^{+}_{\Q(\sqrt{-3 d_K})}$, where $C_{d_K}
=(3,0,\frac{d_{K}}{4})$ or $C_{d_K} =(3,3,\frac{d_{K}+3}{4})$ in
accordance with whether $d_k \equiv 0 \pmod 4$ or $d_k \equiv 1
\pmod 4$.
\end{theorem}

\begin{proof}
We work out the case when $d_k \equiv 1 \pmod 4$, the other case
being completely analogous. By Arndt's composition algorithm (see
\cite{buell}, Theorem 4.10)

$$
\begin{cases}
C_{K}*(P,Q,R)=(P/3, Q, 3R) & \textrm{if $b \equiv 0 \pmod{3}$}\\
C_{K}*(3P, Q, R/3)=(P,Q,R) & \text{if $c \equiv 0 \pmod{3}$}\\
C_{K}*(3P, 2P-Q, \frac{P+R-Q}{3})=(P,2P-Q,P+R-Q) &\text{if $ b \equiv -c \pmod{3}$}\\
C_{K}*(3P, 2P+Q, \frac{P+Q+R}{3}) =(P,2P+Q,P+R+Q)& \text{if $b
\equiv c \pmod{3}.$}
\end{cases}$$
Using the matrix $\left(
\begin{array}{cc}
  1 & -1 \\
0 & 1
\end{array}
\right)$, we see that we have identities in $\Cl^{+}_{\Q(\sqrt{-3
d_K})}$ $$(P,2P-Q,P+R-Q)= H_{K}^{-1} \quad\text{and}\quad
(P,2P+Q,P+R+Q)= H_{K}.$$ Since $C_{K}$ is its own inverse, the
result follows from the explicit description of $\frac{1}{2}q_{K}$
given in Corollary \ref{explicit}.
\end{proof}

\begin{rem}\label{aclaracion}

Note that given $K$ we have freedom in choosing $F_K$ in such a
way that $b \not\equiv -c \pmod{3}$. Hence Theorem \ref{grouprel}
can be actually interpreted as $\frac{1}{2}q_{K}*C_{d_K} =H_{K}$.
\end{rem}

\begin{rem}

We denote the form $C_{K}$ by $C_{d_K}$ in order to stress the
fact that this form only depends on the discriminant of $K$.
\end{rem}

\subsection{Bhargava's composition laws on cubes and their relation to the trace form}

We have related the trace form, in the cubic case, to class groups
of quadratic fields. There is a well known generalization of Gauss'
composition of quadratic forms to cubic forms. Inspired by this
generalization, we expected some connection between the cubic
forms attached to cubic number fields, and the quadratic forms
given by the traces of these fields. We briefly recall some of the
basics of Bhargava's laws on cubes and then we explain how to get
such a connection (see Theorem \ref{grouprel2}).\\

In his PhD thesis (see \cite{bhargava}), Bhargava generalizes the
composition laws on binary quadratic forms of a fixed discriminant
$\Delta$ discovered by Gauss. Bhargava defines a $\SL_2(\Z) \times
\SL_2(\Z) \times \SL_2(\Z)$-action on the set of $2\times 2 \times
2$ integral cubes of discriminant $\Delta$. Let
$\Cl(\Z^{2}\otimes\Z^{2}\otimes\Z^{2}; \Delta)$ be the space of
orbits given of action. Using the generalization of Gauss'
composition mentioned above, Bhargava discovered a composition law
on $\Cl(\Z^{2}\otimes\Z^{2}\otimes\Z^{2}; \Delta)$.

In explicit terms one can think of a $2\times 2 \times 2$ integral
cube $\mathcal{C}$ as a pair of $2\times 2$ integral matrices
$(A,B)$, where $A$ is the front face and $B$ is the back face. Let
$Q_1(\mathcal{C})=-\det(Ax+By)$, $Q_2(\mathcal{C})=-\det(A\left[
\begin{array}{c}
  x \\
  y \\
\end{array}
\right]| B\left[
\begin{array}{c}
  x \\
  y \\
\end{array}
\right])$ and $Q_3(\mathcal{C})=-\det(A^{t}\left[
\begin{array}{c}
  x \\
  y \\
\end{array}
\right]| B^{t}\left[
\begin{array}{c}
  x \\
  y \\
\end{array}
\right])$.

It can be verified that $\Disc(Q_1)=\Disc(Q_2)=\Disc(Q_3)$,
moreover this common discriminant $\Delta$ is precisely the
definition of the discriminant of $\mathcal{C}$. If $g:=(g_1, g_3,
g_3) \in \Gamma :=\SL_2(\Z) \times \SL_2(\Z) \times \SL_2(\Z)$,
and $(A,B)$ is a cube, then $g\cdot(A,B):= g_1{g_{3}Ag_{2}^{t}
\choose g_{3}Bg_{2}^{t}}$. This action preserves the discriminant.
Moreover, if $Q_1,Q_2,Q_3$ are primitive forms, one has that $Q_1
* Q_2 * Q_3 = 0$ as elements of $\Cl^{+}_{\Q(\sqrt{\Delta})}$. Conversely, let $(Q_1,Q_2,Q_3)$
be a triple of primitive, binary quadratic forms of discriminant
$\Delta$ such that $Q_1
* Q_2 * Q_3 = 0$. Then there is a unique class on
$\Cl(\Z^{2}\otimes\Z^{2}\otimes\Z^{2}; \Delta)$ giving rise to
$(Q_1,Q_2,Q_3)$ as above. With this in hand, it is simple to
define a composition law on cubes: $(A,B) + (A',B')$ is the cube
that corresponds to the triple
$(Q_1*Q'_{1},Q_2*Q'_{2},Q_3*Q'_{3})$. Furthermore:
\begin{theorem}[Bhargava, \cite{bhargava}]
There is an isomorphism
\begin{eqnarray*}
\phi: \emph{\Cl}(\Z^{2}\otimes\Z^{2}\otimes\Z^{2}; \Delta) &
\rightarrow & \emph{\Cl}^{+}_{\Q(\sqrt{\Delta})} \times
\emph{\Cl}^{+}_{\Q(\sqrt{\Delta})}
\end{eqnarray*}
defined by $(A,B)_{\Gamma}  \mapsto
([Q_{1}]_{\emph{\SL}_2(\Z)},[Q_{2}]_{\emph{\SL}_2(\Z)}).$
\end{theorem}

\begin{defi}
A binary cubic form $f(x,y) \in \Z[x,y]$ is called a
$\textit{Gaussian cubic form}$ if it is of the form
$(a_0,3a_1,3a_2,a_3)$. The set of Gaussian cubic forms is denoted
by $\emph{\Sym}^{3}\Z^{2}$.
\end{defi}

One may naturally associate to a Gaussian cubic form
$f=(a_0,3a_1,3a_2,a_3)$ a triple symmetric cube:\\ \indent \indent
\indent \indent \indent \indent \indent \indent \indent \indent
\indent $f \mapsto$ $\begin{array}{cc}
    & 
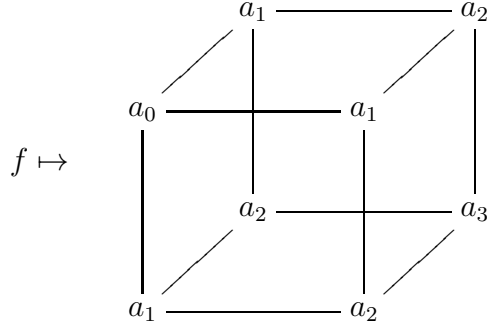
\xymatrix{ & a_1 \ar@{-}[rr]\ar@{-}[dd]
&& a_2 \ar@{-}[dd]\\
a_0 \ar@{-}[ur]\ar@{-}[rr]\ar@{-}[dd]
&& a_1\ar@{-}[ur]\ar@{-}[dd]\\
& a_2\ar@{-}[rr] & & a_3\\
a_1\ar@{-}[rr]\ar@{-}[ur] && a_2\ar@{-}[ur]}\\
\end{array}$

The correspondence between cubic forms and cubes is identified
with a map $\iota: \Sym^{3}\Z^{2} \rightarrow
\Z^{2}\otimes\Z^{2}\otimes\Z^{2}$.

If we replace $f$ by a Gaussian form in the same $\SL_2(\Z)$
equivalence class as $f$, one obtains a well defined element under
the $\Gamma$-action on cubes.

Let $\Cl(\Sym^{3}\Z^{2}; \Delta)$ be the set of Gaussian forms, up
to $\SL_2(\Z)$-action, such that the corresponding cubes have
fundamental discriminant $\Delta$.

\begin{rem}
One must distinguish between the notions of the discriminant of cubic
forms and the discriminant of cubes. For example, let $f$ be a Gaussian
form of discriminant $D$. Then cube  corresponding to $f$ has
discriminant $\Delta =\frac{-D}{27}$.

\end{rem}

It turns out that $\Cl(\Sym^{3}\Z^{2}; \Delta)$ is an abelian
group. Furthermore, we have that
$$[\iota]:[f]_{\SL_{2}(\Z)} \mapsto [\iota(f)]_{\Gamma}$$ is a
group homomorphism. By composing the homomorphisms
$$\Cl(\Sym^{3}\Z^{2}; \Delta) \overset{[\iota]}{\rightarrow} \Cl(\Z^{2}\otimes\Z^{2}\otimes\Z^{2}; \Delta) \overset{\phi}{\rightarrow}
\Cl^{+}_{\Q(\sqrt{\Delta})} \times \Cl^{+}_{\Q(\sqrt{\Delta})}
\overset{\pi_1}{\rightarrow} \Cl^{+}_{\Q(\sqrt{\Delta})},
$$
Bhargava obtains:
\begin{theorem}[Bhargava \cite{bhargava}, Hoffman-Morales
\cite{morales}]\label{morales} There is a surjective homomorphism
$$\phi_{1}: \emph{\Cl}(\emph{\Sym}^{3}\Z^{2}; \Delta) \twoheadrightarrow \emph{\Cl}^{+}_{\Q(\sqrt{\Delta})}[3],$$
where $\phi_{1}$ is the first projection of $\phi$ composed with
$[\iota]$. The cardinality of the kernel is equal to $|U/U^{3}|$,
where $U$ denotes the group of units in $\Q(\sqrt{\Delta}).$ In
other words, the Kernel has order $1$ if $\Delta <-3$, or $3$
otherwise.
\end{theorem}

This theorem was in essence first obtained by Eisenstein, but he
incorrectly asserted that the kernel of the map was always trivial
(see \cite{E}). Later Arnt and Cayley pointed out that it is not a
bijection if $\Delta \geq -3.$

\begin{rem}

Explicitly, $\phi_{1}(a_0,3a_1,3a_2,a_3)=
(a_{1}^2-a_{0}a_{2},a_{1}a_{2}-a_{0}a_{3},a_{2}^2-a_{1}a_{3}).$

\end{rem}

\section{From cubic fields to cubes and trace forms}

Given $K$, a cubic field of discriminant $d_K$, and representative
form $F_K(x,y)=(a,b,c,d)$, we naturally associate a cube as
follows: \\\indent \indent\indent\indent$K: ax^3+bx^2y+cxy^2+dy^3
\longmapsto$ $\begin{array}{cc}
  & \xymatrix{ & b \ar@{-}[rr]\ar@{-}[dd]
&& c \ar@{-}[dd]\\
3a \ar@{-}[ur]\ar@{-}[rr]\ar@{-}[dd]
&& b\ar@{-}[ur]\ar@{-}[dd]\\
& c\ar@{-}[rr] & & 3d\\
b\ar@{-}[rr]\ar@{-}[ur] && c\ar@{-}[ur]}
\end{array}$

We obtain in this way an element $\mathcal{K}_{F} \in
[\iota](\Cl(\Sym^{3}\Z^{2}; -3d_K)) \subseteq
\Cl(\Z^{2}\otimes\Z^{2}\otimes\Z^{2}; -3d_K)$.

Let $D$ be a fundamental discriminant . Let $\mathcal{C}_D \in
\Cl(\Z^{2}\otimes\Z^{2}\otimes\Z^{2}; -3D)$ be given by\\
\indent \indent\indent\indent$\begin{array}{cc}
   & \xymatrix{ & \tiny{\frac{D+3}{4}}\ar@{-}[rr]\ar@{-}[dd]
&& 3\ar@{-}[dd]\\
0\ar@{-}[ur]\ar@{-}[rr]\ar@{-}[dd]
&& 3\ar@{-}[ur]\ar@{-}[dd]\\
& 0\ar@{-}[rr] & & -1\\
1\ar@{-}[rr]\ar@{-}[ur] && 0\ar@{-}[ur]} \\
\end{array}$
or $\begin{array}{cc}
   & \xymatrix{ & \tiny{\frac{D}{4}}\ar@{-}[rr]\ar@{-}[dd]
&& 3\ar@{-}[dd]\\
0\ar@{-}[ur]\ar@{-}[rr]\ar@{-}[dd]
&& 3\ar@{-}[ur]\ar@{-}[dd]\\
& 0\ar@{-}[rr] & & -1\\
1\ar@{-}[rr]\ar@{-}[ur] && 0\ar@{-}[ur]} \\
\end{array}$\\
in accordance with whether $D \equiv 0 \pmod 4$ or $D \equiv 1
\pmod 4$.

\begin{lemma}
Let $K$ be a cubic field with a fixed cubic form $F=(a,b,c,d)$.
Then $Q_1(\mathcal{K}_{F})=H_F$ and
$Q_1(\mathcal{C}_{d_K})=C_{d_{K}}.$
\end{lemma}
\begin{proof}
The result follows easily using the definition
$Q_1(A,B)=-\det(Ax+By)$ for a cube $(A,B)$.
\end{proof}

\begin{theorem}\label{grouprel2}
Let $K$ be a cubic field with discriminant $d_K$ and associated
cubic form $F_K=(a,b,c,d)$. Assume that $d_K$ is fundamental and
that $3$ does not ramify. Let
$T_{F_K}=\mathcal{K}_{F}+\mathcal{C}_{d_{K}}$. Then $(\pi_{1}
\circ \phi)(T_{F_K}))^{\pm1}=\frac{1}{2}q_{K}$ as elements of
$\emph{\Cl}^{+}_{\Q(\sqrt{-3d_K})}$.
\end{theorem}
\begin{proof}
Since $\phi$ is a group homomorphism we have that $\phi(T_{F_{K}})
=\phi(\mathcal(K)_{F})*\phi(\mathcal{C}_{d_{K}}).$ Projecting to
the first component by $\pi_1$ we get that $(\pi_{1} \circ
\phi)(T_{{F}_K}))= H_{K}*C_{d_{K}}.$ Since all of the functions
involved are group homomorphisms, the result follows from Theorem
\ref{grouprel}. In other ``words" \\ $\begin{array}{cc}
   & \xymatrix{ & b \ar@{-}[rr]\ar@{-}[dd]
&& c \ar@{-}[dd]\\
3a \ar@{-}[ur]\ar@{-}[rr]\ar@{-}[dd]
&& b\ar@{-}[ur]\ar@{-}[dd]\\
& c\ar@{-}[rr] & & 3d\\
b\ar@{-}[rr]\ar@{-}[ur] && c\ar@{-}[ur]}  \\
\end{array}$ $+$
$\begin{array}{cc}
   & \xymatrix{ & \tiny{\frac{D+3}{4}}\ar@{-}[rr]\ar@{-}[dd]
&& 3\ar@{-}[dd]\\
0\ar@{-}[ur]\ar@{-}[rr]\ar@{-}[dd]
&& 3\ar@{-}[ur]\ar@{-}[dd]\\
& 0\ar@{-}[rr] & & -1\\
1\ar@{-}[rr]\ar@{-}[ur] && 0\ar@{-}[ur]} \\
\end{array}$ $\begin{array}{c}
 \phi_1 \\
 \longmapsto
 \end{array}$
$\begin{array}{c}
  \\
 \frac{1}{2}\tr_K(x^2)
\end{array}$
\end{proof}

\begin{rem}

We note that we could choose $F_K$, see Remark \ref{aclaracion}, so that the conclusion of Theorem \ref{grouprel2} is
 $(\pi_{1}
\circ \phi)(T_{F_K}))=\frac{1}{2}q_{K}$.
\end{rem}

\begin{theorem}\label{caso3}
Let $K$ be a cubic field with discriminant $d_K$, and let
$F_{K}(x,y)=(a,b,c,d)$ be a cubic form associated to $K$. Assume
that $d_K$ is fundamental and that $3$ ramifies in $K/\Q$. Then we
have:
\begin{eqnarray*}
\phi_{1}: \emph{\Cl}(\emph{\Sym}^{3}\Z^{2}; -\frac{d_{K}}{3}) & \rightarrow & \emph{\Cl}^{+}_{\Q(\sqrt{-\frac{d_{K}}{3}})}[3] \\
                   (f_{K})_{\emph{\SL}_{2}(\Z)}                        & \mapsto &
                   (\frac{1}{6}q_{K})_{\emph{\SL}_{2}(\Z)},
\end{eqnarray*}
\noindent where $f_{K}(x,y)$ is defined as follows:
$$
f_{K}(x,y)=\begin{cases}
\frac{1}{3}F(x,3y) & \textrm{if $b \equiv 0 \pmod{3}$}\\
\frac{1}{3}F(3x,y) & \text{if $c \equiv 0 \pmod{3}$}\\
\frac{1}{3}F(x,3(y-x)) &\text{if $ b \equiv -c \pmod{3}$}\\
\frac{1}{3}F(x,3(y+x)) & \text{if $b \equiv c \pmod{3}.$}
\end{cases} $$

\end{theorem}

\begin{proof}
Replacing $F(x,y)$ with either $F(y,x)$, $F(x,y-x)$ or $F(x,y+x)$
we may assume that $b \equiv 0 \pmod{3}$. With this in hand, we
have that $d_{K} \equiv -ac^{3} \pmod{3}$, and since $3$ ramifies,
$ac \equiv 0 \pmod{3}$. On the other hand since $d_{K}$ is
fundamental we see that $3|a$. By Corollary \ref{explicit},
$\frac{1}{2}q_{K}=((b^2-3ac)/3, bc-9ad, 3(c^2-3bd))$, thus
$\frac{1}{6}q_{K}=((\frac{b}{3})^2-\frac{a}{3}c,
\frac{b}{3}c-\frac{a}{3}9d, (c^2-\frac{b}{3}9d))$, which is
$\phi_{1}(\frac{1}{3}F(x,3y))$.
\end{proof}

\begin{theorem}\label{principal}
Let $K$ be a cubic number field of positive, fundamental
discriminant, and let $L$ be a number field such there exists an
isomorphism of quadratic modules
\begin{eqnarray*}
  \langle O^{0}_{K} , q_{K} \rangle & \cong & \langle O^{0}_{L},q_{L} \rangle.
\end{eqnarray*}
Further assume  $9 \nmid d_{L}$. Then $K \cong L$.
\end{theorem}

\begin{proof}

By Lemma \ref{trace0refinement} we have $d_{K}=d_{L}$. As usual,
fix cubic forms $F_{K}(x,y)$ and $F_{L}(x,y)$ in the classes given
by $K$ and $L$ respectively. Suppose first that $3 \nmid d_{K}$.\\
Since the isometry between the forms need not to be proper, we
only can ensure that as elements of $\Cl^{+}_{\Q(\sqrt{-3 d_K})}$,
$\frac{1}{2}q_K=(\frac{1}{2}q_L)^{\pm 1}$.  By Theorem
\ref{grouprel2} we have that $(\pi_{1} \circ
\phi)(T_{F_K}))^{\pm1}=(\pi_{1} \circ \phi)(T_{F_L}))$.  Replacing
$F_{K}(x,y)$ by $F_{K}(x,-y)$ has the effect of replacing
$H_{F_{K}}(x,y)$ by $H_{F_{K}}(x,-y)$. On the other hand
$H_{F_{K}}(x,-y)$ is inverse to $H_{F_{K}}$ in the narrow class
group. Since $C_{d_{K}}$ has order $2$, Theorem \ref{grouprel}
says that we may replace $F_{K}(x,y)$ by $F_{K}(x,-y)$, if
necessary, so we may assume that
$$(\pi_{1} \circ \phi)(T_{F_K})=(\pi_{1} \circ \phi)(T_{F_L}).$$
Equivalently,
$$(\pi_{1} \circ \phi)(\mathcal{K}_{F_{K}})=(\pi_{1} \circ \phi)(\mathcal{K}_{F_{L}}).$$
Notice that $\mathcal{K}_{F}=\iota(3F)$, hence
$\phi_1(3F_{K})=\phi_1(3F_{L})$. Since $d_K >1$, Theorem
\ref{morales} implies that $3F_{K}$ and $3F_{L}$ are
$\SL_2(\Z)$-equivalent. Since we could have replaced $F_{K}(x,y)$
by $F_{K}(x,-y)$, the equivalence between $3F_{K}$ and $3F_{L}$ is
up to $\GL_2(\Z)$. In any case this implies that $K \cong L$. If
$3 \mid d_{K}$, we apply Theorem \ref{caso3} and the argument
follows the same lines as in the case without $3$-ramification.
\end{proof}

\subsection{Observations.}

Given $\Delta \in \Z$, let $X_{\Delta}$ be the set of integral,
primitive, binary quadratic forms of discriminant $\Delta$. Recall
our notation $\Gamma_{\Delta} = \text{GL}_{2}(\Z)\setminus
X_{\Delta}$ and $\Gamma_{\Delta}^{1}=\text{SL}_{2}(\Z) \setminus
X_{\Delta}$.

Let $d$ be a positive fundamental discriminant, $n_d:=$gcd$(3,d)$,
and $\mathcal{C}_{d}$ the set of isomorphism classes of cubic
fields of discriminant $d$.

\begin{rem}\label{equivalencias} Theorem \ref{principal} is equivalent to the injectivity of
\begin{eqnarray*}
\Phi_{d} : \mathcal{C}_{d} & \longrightarrow & \Gamma_{\frac{-3d}{n_{d}^{2}}} \\
     K & \mapsto & [\frac{1}{2n_d
}q_{K}].
\end{eqnarray*}
\end{rem}

Recall that Gauss' composition induces a group isomorphism between
$\Cl^{+}_{\Q(\sqrt{\frac{-3d}{n_{d}^{2}}})}$ and
$\Gamma^{1}_{\frac{-3d}{n_{d}^{2}}}$. Hence, we have a double
cover $\pi : \Cl^{+}_{\Q(\sqrt{\frac{-3d}{n_{d}^{2}}})}
\rightarrow\Gamma_{\frac{-3d}{n_{d}^{2}}}$, with the property that
the fiber of every point consists of an element and its inverse.
Therefore, even though $\frac{1}{2n_d }q_{K}$ does not define a
point in $\Cl^{+}_{\Q \left(\sqrt{\frac{-3d}{n_{d}^{2}}} \right)
}$, it defines a cyclic subgroup, namely the group generated by
$\pi^{-1}(\Phi_{d}(K))$. Corollary \ref{explicit} and Lemma
\ref{caso3} provide us with a generator of this group. Let $g_k$
be such a generator. Using Arndt's composition algorithm (see
\cite{buell}), one sees that $g_{K}^{3}=C_{K}$ when $3 \nmid d$,
and that $g_{K}$ has order $3$ otherwise. Since $C_{d_K}$ has
order $2$, it follows that $\langle \pi^{-1}(\Phi_{d}(K))\rangle $
has order $2n_{d}$.

\begin{proposition}\label{ordertrace}
Let $d>0$ be a fundamental discriminant. The map $K
\mapsto \langle g_K \rangle$ is injective.
\end{proposition}

\begin{proof}

Since $\langle g_K\rangle$ has order $3$ or $6$, its set of
generators is $\{g^{\pm 1}_{K} \}$. Thus, if $\langle g_K
\rangle$=$\langle g_L \rangle$, then $g^{\pm 1}_{K}=g_{L}$.
Projecting under $\pi$ we obtain that $\Phi_{d}(K)=\Phi_{d}(L)$,
and the result follows from Remark \ref{equivalencias}.
\end{proof}
Note that the unique subgroup of order $3$ of $\langle g_K\rangle$
is given by $\langle g^{2}_K \rangle$. Hence, from
Proposition \ref{ordertrace} we have:

\begin{theorem}\label{ultimo}
Let $d>0$ be a fundamental discriminant such that $\mathcal{C}_{d}
\neq \emptyset$. Let
$\mathcal{P}_{3}(\emph{\Cl}^{+}_{\Q(\sqrt{-3d})})$ be the set of
subgroups of size $3$ of $\emph{\Cl}^{+}_{\Q(\sqrt{-3d})}$. Then
\begin{eqnarray*}
\Theta_{d} : \mathcal{C}_{d} & \longrightarrow & \mathcal{P}_{3}(\emph{\Cl}_{\Q(\sqrt{-3d})}) \\
     K & \mapsto & \langle g_{K}^{2} \rangle
\end{eqnarray*}
is injective.

\end{theorem}

The injection  $\Theta_{d}$ gives an alternative proof of one
inequality of the Scholz Reflection Principle (see \cite{Scholz}).

\begin{cor}\label{scholz}
Let $d$ be a positive fundamental discriminant, and let
$r=\emph{r}_{3}(-3d)$ and $s=\emph{r}_{3}(d)$ (recall our notation
$\emph{r}_{3}(d)=\emph{\dim}_{\mathbb{F}_{3}}(\emph{\Cl}_{\Q(\sqrt{d})}
\otimes_{\Z} \mathbb{F}_3)$). Then $s \leq r$.
\end{cor}

\begin{proof}
$(3^{s}-1)/2=|\mathcal{C}_{d}|$ and
$(3^{r}-1)/2=|\mathcal{P}_{3}(\Cl_{\Q(\sqrt{-3d})})|$.
\end{proof}

\section*{Acknowledgements}
I would like to thank Jordan Ellenberg for introducing me to this
subject, and for many helpful discussions, suggestions and
advice during the writing of this paper. I also thank Manjul Bhargava, Amanda
Folsom, and Yongqiang Zhao for thorough and helpful comments on an
earlier version of this paper.

\end{document}